\documentclass[12pt,a4paper]{amsart}
\usepackage[centertags]{amsmath}
\usepackage{amsfonts}
\usepackage{amssymb}
\usepackage{amsthm}
\usepackage{newlfont}

\usepackage{amsbsy,amscd,amsmath,amsopn,amssymb,amsthm}
\usepackage{color} 
\usepackage{color, graphics, epic, eepic}

\theoremstyle{definition}
\newtheorem{thm}{Theorem}[section]
\newtheorem{lem}[thm]{Lemma}
\newtheorem{prp}[thm]{Proposition}

\newtheorem{cor}[thm]{Corollary}

\newtheorem{exa}[thm]{Example}

\newcommand{\Tr}{\operatorname{Tr}}

\begin{document}
\title[Characterization of the monotonicity]
{Characterization of the monotonicity by the inequality}


\author[D.~T.~Hoa]{Dinh Trung Hoa}
\thanks{}
\address{{
Center of Research and Development, Duy Tan University, K7/25 Quang Trung, Danang, Vietnam}}
\email{dinhtrunghoa@duytan.edu.vn}

\author[H. Osaka]{HIROYUKI OSAKA$^{a}$} 
\date{15, July, 2012}
\thanks{}
\address{Department of Mathematical Sciences, Ritsumeikan
University, Kusatsu, Shiga 525-8577, Japan}

\email{osaka@se.ritsumei.ac.jp}

\author[J. Tomiyama]{JUN TOMIYAMA}
\thanks{}
\address{Prof. Emeritus of Tokyo Metropolitan University,
201 11-10 Nakane 1-chome, Meguro-ku, Tokyo, Japan}

\email{juntomi@med.email.ne.jp}





\footnote{$^A$Research partially supported by the JSPS grant for 
Scientic Research No. 20540220.}
\begin{abstract}
Let $\varphi$ be a normal state on the algebra $B(H)$ of all bounded operators on a Hilbert space $H$, $f$ a strictly positive, continuous function on $(0, \infty)$, 
and let $g$ be a function on $(0, \infty)$ defined by $g(t) = \frac{t}{f(t)}$.
We will give characterizations of matrix and operator monotonicity by the following generalized Powers-St\o rmer inequality: 

$$
\varphi(A + B) - \varphi(|A - B|) \leq 2\varphi(f(A)^\frac{1}{2}g(B)f(A)^\frac{1}{2}),
$$
whenever $A, B$ are positive invertible operators in $B(H).$
\end{abstract}
\maketitle

\section{Introduction}
Throughout the paper, $M_n$ stands for the algebra of all $n\times n$ matrices, $M_n^+$ denote the set of positive semi-definite matrices. We call a function $f$ \textit{matrix convex of order $n$} or 
\textit{$n$-convex} in short (resp. \textit{matrix concave of order} $n$ or \textit{$n$-concave})
whenever the inequality 
$$
f(\lambda A + (1 - \lambda)B) \leq \lambda f(A) + (1 - \lambda) f(B), \ 
 \lambda \in [0, 1]
 $$ 
(resp. $
f(\lambda A + (1 - \lambda)B) \geq \lambda f(A) + (1 - \lambda) f(B), \ 
 \lambda \in [0, 1]
$) 
holds for every pair of selfadjoint matrices $A, B \in M_n$ such that 
all eigenvalues of $A$ and $B$ are contained in $I$.
\textit{Matrix monotone functions} on $I$ are 
similarly defined as
the inequality 
$$
A \leq B \Longrightarrow f(A) \leq f(B)
 $$
for any pair of selfadjoint matrices $A, B \in M_n$ such that 
$A \leq B$ and all eigenvalues of $A$ and $B$ are contained in $I$.
We call a function $f$ \textit{operator convex} (resp. \textit{operator concave}) if for each $k \in \mathbb{N}$, $f$ is 
$k$-convex (resp. $k$-concave) and \textit{operator monotone} if for each $k \in \mathbb{N}$ $f$ is 
$k$-monotone.

Let $n \in \mathbb{N}$ and $f: [0, \alpha) \rightarrow \mathbb{R}$. 
In \cite{OJ} the second and the third author discussed about the following 3 assertions 
at each level $n$ among them in order to see clear insight of 
the double piling structure of matrix monotone functions and of matrix convex functions:
\begin{enumerate}
\item[(i)] $f(0)\leq 0$ and $f$ is $n$-convex in $[0,\alpha)$,
\item[(ii)] For each matrix $a$ with its spectrum in $[0,\alpha)$ and a contraction $c$ in the matrix algebra
$M_n$,
 \[f(c^{\star}a c)\leq c^{\star}f(a)c,\]
\item[(iii)] The function $\frac{f(t)}{t}$ $(= g(t))$ is $n$-monotone in $(0,\alpha)$.
\end{enumerate}
It was shown in \cite{OJ} that
$$
(i)_{n+1} \prec (ii)_n \sim (iii)_n \prec  (i)_{[\frac{n}{2}]},
$$
where denotion $(A)_m \prec (B)_n$ means that ``if $(A)$ holds for the matrix algebra $M_m$, 
then $(B)$ holds for the matrix algebra $M_n$''.

In this article, using an idea in \cite{HOT} we can get the concave version of the above observation. 
Namely, for $n \in \mathbb{N}$ and $f: [0, \alpha) \rightarrow \mathbb{R}$ we consider the following assertions:
\begin{enumerate}
\item[(iv)] $f(0)\geq 0$ and $f$ is $n$-concave in $[0,\alpha)$,
\item[(v)] For each matrix $a$ with spectrum in $[0,\alpha)$ and a contraction $c$ in the matrix algebra
$M_n$,
 \[f(c^{\star}a c)\geq c^{\star}f(a)c,\]
\item[(vi)] The function $\frac{t}{f(t)}$ is $n$-monotone in $(0,\alpha)$.
\end{enumerate}

We will show that 
$$
(iv)_{n+1} \prec (v)_n \sim (vi)_n \prec  (iv)_{[\frac{n}{2}]}.
$$

As an application we investigate the generalized Powers-St\o rmer inequality 
from the point of matrix functions, which was introduced in \cite{HOT}. Let $\varphi$ be a normal state on the algebra $B(H)$ of all bounded operators on a Hilbert space $H$, 
$f$ be a strictly positive, continuous function on $(0, \infty)$, 
and let $g$ be a function on $(0, \infty)$ defined by $g(t) = \frac{t}{f(t)}$.
We will consider the following inequality 
$$
\varphi(A + B) - \varphi(|A - B|) \leq 2\varphi(f(A)^\frac{1}{2}g(B)f(A)^\frac{1}{2}),
$$
where $A, B$ are positive invertible operators in $B(H).$

It will be shown that:
\begin{enumerate}
\item If the inequality holds true for any positive invertible $A$, $B$, 
then the function $g$ is operator monotone.
\item When $\dim H = n < \infty$, if $\varphi$ is canonical trace and $f$ is $(n + 1)$-concave, then the inequality holds. 
\item When $\dim H = n < \infty$, if the inequality holds, then the state $\varphi$ has the trace property if and only if the function $g$ satisfies the condition
$$\inf_{\lambda > \mu}
\dfrac{\sqrt{g'(\lambda)g'(\mu)}}
{\dfrac{g(\lambda) - g(\mu)}{\lambda - \mu}} =0.$$
\end{enumerate}

\section{Hansen-Pedersen's inequality for matrix functions}

For a long time it has been known the following equivalency.
When $f$ is strictly positive, continuous function on $(0, \infty)$, the followings  are equivalent (\cite[2.6. Corollary]{Hansen-Perd-annal}):

\begin{enumerate}
\item $f$ is operator concave.
\item $\dfrac{t}{f(t)}$ is operator monotone
\end{enumerate}

\vskip 3mm

The following result is the matrix function versions of the above observation.

\vskip 3mm

\begin{thm}\label{thm:Hansen-Pedersen}
Let $n \in \mathbb{N}$ and $f \colon [0, \alpha) \rightarrow \mathbb{R}$ be a continuous function for some $\alpha > 0$ 
such that $0 \notin f([0, \alpha))$.
Let us consider the following assertions:
\begin{enumerate}
\item[$(4)_n$] $f$ is $n$-concave with $f(0) \geq 0$.
\item[$(5)_n$] For all operators $A \in M_n$ with its spectrum in $[0, \alpha)$ 
and all contraction $C$
$$
f(C^*AC) \geq C^*f(A)C.
$$
\item[$(6)_n$] $g(t) = \frac{t}{f(t)}$ is $n$-monotone on $(0, \alpha)$.
\end{enumerate}

Then we have
$$
(4)_{n+1} \prec (5)_n \sim (6)_n \prec  (4)_{[\frac{n}{2}]}.
$$
\end{thm}

\vskip 3mm

\begin{proof}
The implication $(4)_{n+1} \prec (5)_n$:

Since $f$ is $(n + 1)$-concave, $-f$ is $(n + 1)$-convex. From \cite{OJ} we know that 
for an operator $A \in M_n$ with its spectrum in $[0, \alpha)$ and a contraction $C$  
$$
(-f)(C^*AC) \leq C^*(-f)(A)C,
$$
hence 
$$
f(C^*AC) \geq C^*f(A)C.
$$

The implication $(5)_{n} \sim (6)_n$:

Since $f(C^*AC) \geq C^*f(A)C$ for an operator $A \in M_n$
 with its spectrum in $[0, \alpha)$ and a contraction $C$, we know 
 that 
$$
(-f)(C^*AC) \leq C^*(-f)(A)C.
$$
Then, from \cite{OJ} we know that $\frac{-f(t)}{t}$ is $n$-monotone. 
Since the function $-\dfrac{1}{t}$ is operator monotone, $-\dfrac{1}{\dfrac{-f(t)}{t}}$ is $n$-monotone, 
that is, $\frac{t}{f(t)}$ is $n$-monotone.

Conversely, if $\frac{t}{f(t)}$ is $n$-monotone, then $-\frac{1}{\dfrac{t}{f(t)}} = \dfrac{(-f)(t)}{t}$ is 
$n$-monotone from the operator monotonicity of $-\dfrac{1}{t}$, hence we know in \cite{OJ} that for an operator $A \in M_n$
 with its spectrum in $[0, \alpha)$ and a contraction $C$
$$
(-f)(C^*AC) \leq C^*(-f)(A)C, 
$$
that is,
$$
f(C^*AC) \geq C^*f(A)C.
$$

The implication $(6)_{n} \prec (4)_{[\frac{n}{2}]}$:
Since $\frac{t}{f(t)}$ is $n$-monotone, the function $\frac{-f(t)}{t}$ is 
$n$-monotone by \cite[Lemma~2.2]{HOT}. Hence, $-f$ is $[\frac{n}{2}]$-convex by \cite{OJ}, that is, 
$f$ is $[\frac{n}{2}]$-concave.
\end{proof}

\section{Characterization of matrix monotonicity}

The following result was proved in \cite{HOT} under the condition that the function $f$ 
is $2n$-monotone. But using the concavity  
we 
will show 
that the condition of $f$ is weakened.
Note that the $2n$-monotonicity of a function $f$ on $[0, \infty)$ implies the 
$n$-concavity of $f$ by \cite[Theorem~V.2.5]{RB}.

\vskip 3mm

\begin{thm}\label{thm:Powers-Stormer}
Let $\Tr$ be the canonical trace on $M_n$ and $f$ be a $(n + 1)$-concave 
function on $[0, \infty)$ such that $f((0, \infty)) \subset (0, \infty)$. 
Then for any pair of positive matrices $A, B \in M_n$  
$$
\Tr(A) + \Tr(B) - \Tr(|A - B|) 
\leq 2\Tr(f(A)^{\frac{1}{2}}g(B)f(A)^{\frac{1}{2}}),
$$
where 
$g(t) 
= 
\left\{\begin{array}{cc}
\frac{t}{f(t)}&(t \in (0, \infty))\\
0& (t = 0)
\end{array}
\right.
$.
\end{thm}

\vskip 3mm

\begin{proof}
Since $f$ is $(n + 1)$-concave, we know that the function $g$ is $n$-monotone 
from Theorem~\ref{thm:Hansen-Pedersen}. 
Moreover, by a standard argument (see, for example, \cite[Theorem~V.2.5]{RB}) it is clear that the function $f$ is $(n + 1)$-monotone, and hence $n$-monotone.
Repeating the similar argument as in the proof of \cite[Theorem~2.1]{HOT} 
with mentioned properties of $f$ and $g$, we will get the conclusion.
\end{proof}
\vskip 3mm

From the above Theorem \ref{thm:Powers-Stormer} we consider the following converse problems.

\vskip 3mm

Let $n \in \mathbb{N}$ and $\varphi$ be a faithful
positive linear functional on $M_n$, $f$ be a strictly positive, continuous function on $(0, \infty)$, 
and let $g$ be a function on $(0, \infty)$ defined by $g(t) = \frac{t}{f(t)}$.
Suppose that for any positive invertible $A, B \in M_n$
\begin{equation}\label{1}
\varphi(A + B) - \varphi(|A - B|) \leq 2\varphi(f(A)^\frac{1}{2}g(B)f(A)^\frac{1}{2}).
\end{equation}

Then:

\textbf{Problem 1:} Is it true that $f$ is $n$-monotone?

\textbf{Problem 2:} Is $\varphi$ a scalar positive multiple of the canonical trace?

\vskip 3mm

The following examples give a contribution to the attempt to answer problem 1.

\vskip 3mm

\begin{exa}
Let $f(t) = t^2$ on $(0, \infty)$. 
It is well-known that $f$ is not 2-monotone.
We now show that the function $f$ does not satisfy the inequality $(\ref{1})$. 
Indeed, let us consider the following matrices
$$
A = \left(\begin{array}{cc}
               1&1\\
               1&1
            \end{array}
          \right)
\quad \hbox{and} \quad  B = \left(\begin{array}{cc}
               2&1\\
               1&2
            \end{array}
          \right).
          $$

Then we have 
$$AB^{-1}A = \frac{2}{3}A.$$

Set $\tilde{A} = A \oplus \mathrm{diag}(\underbrace{1,\cdots, 1}_{n-2}), 
\tilde{B} = B \oplus \mathrm{diag}(\underbrace{1,\cdots, 1}_{n-2})$ in $M_n$. 
Then, $\tilde{A} \leq \tilde{B}$ and for a faithful linear functional $\varphi$ on $M_n$

\begin{align*}
\varphi(f(\tilde{A})^\frac{1}{2}g(\tilde{B})f(\tilde{A})^\frac{1}{2}) 
&= \varphi(\tilde{A}\tilde{B}^{-1}\tilde{A})\\
&= \varphi(\frac{2}{3}A \oplus \mathrm{diag}(\underbrace{1,\cdots, 1}_{n-2}))\\
&< \varphi(A \oplus \mathrm{diag}(\underbrace{1,\cdots, 1}_{n-2})) \\
&= \varphi(\tilde{A}).
\end{align*}

On the contrary, since $\tilde{A} \leq \tilde{B}$, from the inequality (\ref{1}) we have 
$$
\varphi(\tilde{A}) + \varphi(\tilde{B}) - \varphi(\tilde{B} - \tilde{A}) 
\leq 2\varphi(f(\tilde{A})^\frac{1}{2}g(\tilde{B})f(\tilde{A})^\frac{1}{2}),
$$
or
$$\varphi(\tilde{A}) \leq \varphi(f(\tilde{A})^\frac{1}{2}g(\tilde{B})f(\tilde{A})^\frac{1}{2}),
$$
and we have a contradiction.
\hfill $\qed$
\end{exa}

\vskip 3mm


Now we will show that for $p > 1$ the function $f(t) = t^p$ 
does not satisfy the special inequality from inequality $(\ref{1})$ 
for a faithful positive linear functional.

\vskip 3mm

\begin{exa}\label{exa:tp}
It will be shown that for $f(t) = t^p\ (p > 1)$ and 
a faithful positive linear functional $\varphi$ on $M_n$ the following inequality does not always hold:
\begin{equation}\label{2}
\varphi (A) \le \varphi (f(A)^{1/2} g(B) f(A)^{1/2})\quad (0 \le A \le B).
\end{equation} 
Note that when $0 < A \leq B$ the inequality $(\ref{2})$ can be deduced from the inequality $(\ref{1})$ directly.

Since $A \le B$, we can suppose that $B=A+tC$ for some positive number $t$ and positive matrix $C$. Hence inequality (\ref{1}) becomes
\begin{equation}\label{3}
\varphi(A) \leq \varphi(f(A)^\frac{1}{2}g(A+tC)f(A)^\frac{1}{2}).
\end{equation}
On the other hand, we have
\begin{equation}\label{4}
g(A+tC)=g(A) + t \cdot \frac{d g(A+tC)}{dt}\big |_{t=0} + R(A,C,t),
\end{equation}
where $\lim_{t \rightarrow 0} \frac{||R(A,C,t)||}{t}=0.$

From (\ref{3}) and (\ref{4}) we get
\begin{align*}
\varphi(A) & \le \varphi (f(A)^{1/2} (g(A) + t \cdot \frac{d g(A+tC)}{dt}\big |_{t=0} + R(A,C,t)) f(A)^{1/2}) \\
& = \varphi(A) + t\varphi (f(A)^{1/2} \cdot \frac{d g(A+tC)}{dt}\big |_{t=0} \cdot f(A)^{1/2}) + o(t).
\end{align*}
From that, we get
\begin{equation}\label{5}
\varphi (f(A)^{1/2} \cdot \frac{d g(A+tC)}{dt}\big |_{t=0} \cdot f(A)^{1/2}) \geq 0 \quad (\forall A, C \geq 0).
\end{equation}

Let us assume that $\varphi(\cdot)=\Tr(S\cdot),$ where $S= \mathrm{diag} (s,1)~(s \in
[0,1]).$ For $\beta > 0$ and $\alpha \in [0,1]$, let us consider
the $2 \times 2$ matrices 
$$
C=\left(%
\begin{array}{cc}
  \alpha^2 & \alpha\sqrt{1-\alpha^2} \\
  \alpha\sqrt{1-\alpha^2} & 1-\alpha^2 \\
\end{array}
\right) \quad \hbox{and}\quad A=\beta P_1+ \alpha P_2,
$$
where
$$
P_1=\left(%
\begin{array}{cc}
  1 & 0 \\
  0 & 0 \\
\end{array}%
\right), \quad P_2=\left(%
\begin{array}{cc}
  0 & 0 \\
  0 & 1 \\
\end{array}%
\right).
$$

Then by identifying $M_2$ as a $C^*$-subalgebra $M_2 \oplus O_{n-2}$ of $M_n$,
we may assume that $\varphi(\cdot)=\Tr(S\cdot),$ where $S= \mathrm{diag} (s,1)~(s \in
[0,1]).$

Then (\ref{5}) becomes

\begin{equation}\label{6}
s \alpha^2 (1 -\frac{\beta f'(\beta)}{f(\beta)}) + (1-\alpha^2) (1-\frac{\alpha f'(\alpha)}{f(\alpha)}) \ge 0.
\end{equation}
Since $f(t)=t^p$, for any $\alpha, \beta > 0$ we have 
$$ 1 -\frac{\beta f'(\beta)}{f(\beta)} = 1-\frac{\alpha f'(\alpha)}{f(\alpha)} = 1 - p < 0.$$

It implies that the latter inequality (\ref{6}) does not hold true for any $\alpha \in [0,1]$, and the inequality (\ref{2}) will not hold true.

In particular, even $\varphi$ is the canonical trace on $M_n$, the inequality (\ref{2}) will not hold true.

From above argument, we can find many counterexamples for the functions not of the form $f(t)=t^p\ (p>1).$ 
For example, if function $f$ on some $(0,a)$ satisfies condition $f(t) < tf'(t)$, then inequality (\ref{2}) 
is not true.
\hfill$\qed$
\end{exa}

\vskip 3mm

Here we will give a positive answer on problem 2 for some class of functions $g$, namely, 
it will be shown that inequality~(\ref{2}) characterizes the trace property of $\varphi$.

\vskip 3mm

\begin{prp}
Let $n \in \mathbb{N}$ and $\varphi$ be a faithful positive linear functional on $M_n$.
Let $f$ be a strictly positive, continuous function on $(0, \infty)$, 
and let $g$ be a function on $(0, \infty)$ defined by $g(t) = \frac{t}{f(t)}$.
Suppose that 
\begin{equation}\label{7}
\varphi(A) \leq \varphi(f(A)^\frac{1}{2}g(B)f(A)^\frac{1}{2}),
\end{equation}
whenever any pair of positive invertible $A, B \in M_n$ such that $0 < A \leq B.$

Then $\varphi$ has the trace property if and only if $g$ satisfies
the condition 
\begin{equation}\label{condition}
\inf_{\lambda > \mu}
\frac{\sqrt{g'(\lambda)g'(\mu)}}
{\frac{g(\lambda) - g(\mu)}{\lambda - \mu}} =0.
\end{equation}
\end{prp}

\vskip 3mm

\begin{proof}
The conclusion follows from the same steps in the proof of \cite[Theorem~2.2]{Sano-Yatsu}, but 
we put the sketch of the proof for readers.

Let $S$ be a positive definite matrix such that $\varphi(X) = \Tr(XS)$ ($X \in M_n$). Then the trace property of $\varphi$ is equivalent to the condition that 
$S$ is a positive scalar multiple of the identity matrix. 
Taking into consideration 
$$
\varphi(V^*\cdot V) = \Tr(\cdot VSV^*)
$$
for all unitary $V$ and that $VSV^*$ is diagonal for a unitary $U$, we may assume that
that $\varphi(\cdot)=\Tr(S\cdot),$ where $S=\mathrm{diag} (s,1)~(s \in
[0,1]).$ For $\beta > 0$ and $\alpha \in [0,1]$, let us consider
the matrices
$$U=\frac{1}{\sqrt{2}}\left(%
\begin{array}{cc}
  1 & 1 \\
  1 & -1 \\
\end{array}%
\right),\quad
C=\left(%
\begin{array}{cc}
  \alpha^2 & \alpha\sqrt{1-\alpha^2} \\
  \alpha\sqrt{1-\alpha^2} & 1-\alpha^2 \\
\end{array}
\right)$$ and
$$A= \lambda P_1+ \mu P_2,$$
where
$$
P_1=\left(%
\begin{array}{cc}
  1 & 0 \\
  0 & 0 \\
\end{array}%
\right), \quad P_2=\left(%
\begin{array}{cc}
  0 & 0 \\
  0 & 1 \\
\end{array}%
\right).
$$
And (\ref{7}) becomes

\begin{align*}
\frac{\alpha}{\sqrt{1-\alpha^2}} \frac{f(\lambda)^{1/2}}{f(\mu)^{1/2}}g'(\lambda) 
+ \frac{\sqrt{1-\alpha^2}}{\alpha}\frac{f(\mu)^{1/2}}{f(\lambda)^{1/2}}g'(\mu) 
\ge 2 \frac{1-s}{1+s} \frac{g(\lambda) - 
g(\mu)}{\lambda - \mu}.
\end{align*}
Put 

$$
t= \frac{\alpha}{\sqrt{1-\alpha^2}}  \quad \hbox{and} \quad \delta = \frac{1-s}{1+s}.
$$

Note that $0 < \alpha < 1 \Longleftrightarrow 0 < t < \infty$.

Then it is clear that $s=1$ iff $\delta =0.$ The latter inequality is described as
$$
\frac{1}{2}\big(t\frac{f(\lambda)^{1/2}}{f(\mu)^{1/2}}g'(\lambda)+\frac{1}{t} \frac{f(\mu)^{1/2}}{f(\lambda)^{1/2}}g'(\mu)\big) 
\ge \delta \frac{g(\lambda) - g(\mu)}{\lambda - \mu}.
$$
Hence, by considering arithmetic-geometric mean inequality in the left-hand side, we have
$$
\sqrt{g'(\lambda)g'(\mu)} 
\ge \delta 
\frac{g(\lambda) - g(\mu)}{\lambda - \mu}.
$$
Therefore, the condition that $\varphi$ has the trace property, that is, the condition $s=1$ or $\delta =0$ is given by
$$
\inf_{\lambda > \mu}
\frac{\sqrt{g'(\lambda)g'(\mu)}}
{\frac{g(\lambda) - g(\mu)}{\lambda - \mu}} = 0.
$$
\end{proof}

\vskip 3mm

\begin{exa}
For $g(x) = t^2$ (i.e. $f(t) = 1/t)$ on $(0, \infty)$ which satisfies the condition (\ref{condition}), and for any $n \in \mathbb{N}$  
$$\Tr(A) \leq \Tr(f(A)^\frac{1}{2}g(B)f(A)^\frac{1}{2})$$ 
whenever $0 < A \leq B$ in $M_n$.

Indeed, by assumption we have $B^{-1} \leq A^{-1}$. Consequently,

$$
A \le B = BB^{-1}B \leq BA^{-1}B.
$$
Therefore,
$$\Tr(A) \leq \Tr(BA^{-1}B) =  \Tr(A^{-1/2}B^2A^{-1/2}).
$$
\hfill$\qed$
\end{exa}

\vskip 3mm

We have the following inequality for the exponential functions $g(t) = e^t$ on $(a, \infty)$ which satisfies the condition (\ref{condition}).

\begin{exa}\label{exa:exponential}
For any natural number $n$, we have 
$$
\Tr(A) \leq \Tr((Ae^{-A})^\frac{1}{2} e^B (A e^{-A})^\frac{1}{2})
$$
whenever $0 < A \leq B$ in $M_n$.

Indeed, let $A, B$ in $M_n$ such that $0 < A \leq B$.
Since $0<A$, we have 
$$0 < A e^{- A} \quad \hbox{and} \quad 0 < \log(A e^{- A}).$$ 
Besides, 
$$0 < A \leq B \quad  \Longrightarrow \quad 0 < A \leq \log(A e^{- A}) + B.$$ 
Consequently,
$$
0 <A \leq \log(Ae^{- A}) + B \leq e^{\log(Ae^{- A}) + B}.
$$

On account of Golden-Thompson's Inequality, from the latter inequality it follows 
\begin{align*}
\Tr(A) &\leq \Tr(e^{\log(Ae^{-A}) + B})\\
       &\leq \Tr(e^{\log(Ae^{-A})}e^B) \\
        &\leq \Tr(Ae^{-A} e^B)\\
        &= \Tr((Ae^{-A})^\frac{1}{2} e^B (A e^{-A})^\frac{1}{2}).
\end{align*}

\hfill$\qed$
\end{exa}

On the contrary, when $g(t) = t^3$, the inequality $(\ref{7})$ does not hold always.

\vskip 3mm

\begin{exa}
Let $g(t) = t^3$.
Suppose that the inequality $(\ref{7})$ holds for $0 < A \leq B$ in $M_2$. 
Then we have 
$$
\Tr(A) \leq \Tr(A^{-1}BA^{-1})
$$
for $0 < A \leq B$. 
Set $A = \mathrm{diag}(2,2)$ and $B = A$. Since $A^{-1} = \mathrm{diag}\left(\dfrac{1}{2}, \dfrac{1}{2}\right)$, 
we have
\begin{align*}
4 = \Tr(A) &\leq\Tr(A^{-1}AA^{-1})\\
&= \Tr(A^{-1}) = 1,
\end{align*}
and a contradiction.
\hfill$\qed$.
\end{exa}

\vskip 3mm
\section{Characterization of operator monotonicity}

The following lemma is obvious.

\begin{lem}\label{lem:traceclass}
Let $A = (a_{ij}), B = (b_{ij})$ be positive invertible in $M_n$ and $S$ be the density operator on an infinite dimensional, separable  Hilbert space $H$. Suppose that $a_{11} > b_{11}$. Then there exist an orthogonal system 
$\{\xi_i\}_{i=1}^\infty \subset H$ and $\{\lambda_i\}_{i=1}^\infty \subset [0, 1)$ such that $\sum_{i=1}^\infty\lambda_i = 1$, 
$S\xi_i = \lambda_i\xi_i$, and 
$\sum_{i=1}^na_{ii}\lambda_i > \sum_{i=1}^nb_{ii}\lambda_i$. 
\end{lem}

\vskip 3mm

\begin{thm}\label{thm:monotonicity}
Let $\varphi$ be a normal state on $B(H)$, $f$ be a strictly positive, continuous function on $(0, \infty)$, 
and let $g$ be a function on $(0, \infty)$ defined by $g(t) = \frac{t}{f(t)}$.
Suppose that for any positive invertible $A, B \in B(H)$
$$
\varphi(A + B) - \varphi(|A - B|) \leq 2\varphi(f(A)^\frac{1}{2}g(B)f(A)^\frac{1}{2}).
$$ 
Then the function $g$ on $(0, \infty)$ is operator monotone.
\end{thm}

\begin{proof}
Let $S_\varphi$ be a density operator on $H$ such that $\varphi(X) = \mathrm{Tr}(S_\varphi X)$ for all $X \in B(H)$.

Suppose that $g$ is not operator monotone. We have, then, there exist $n \in \mathbb{N}$ and invertible positive matrices $A, B$ in $M_n$ with $A \leq B$ such that $g(A) \not\leq g(B)$.
Hence, $A \not\leq f(A)^\frac{1}{2}g(B)f(A)^\frac{1}{2}$. 

Let $A = [a_{ij}]$ and $f(A)^\frac{1}{2}g(B)f(A)^\frac{1}{2} = [b_{ij}] = B'$.
Since $S_\varphi$ is a density operator, from Lemma~\ref{lem:traceclass} 
there exist an orthogonal system $\{\xi_i\}_{i=1}^\infty \subset H$ and 
$\{\lambda_i\}_{i=1}^\infty \subset  [0, 1)$ such that $\sum_{i=1}^\infty\lambda_i = 1$ and 
$\sum_{i=1}^na_{ii}\lambda_i > \sum_{i=1}^nb_{ii}\lambda_i$. 

Let $\rho \colon M_n \rightarrow (\sum_{i=1}^n|\xi_i\rangle\langle\xi_i|)B(H)(\sum_{i=1}^n|\xi_i\rangle\langle\xi_i|)$ be a canonical inclusion
defined by $\rho([x_{ij}]) = \sum_{i, j = 1}^nx_{ij}|\xi_i\rangle\langle\xi_j|$. 
Let $C = \rho(A) + \sum_{i=n+1}^\infty|\xi_i\rangle\langle\xi_i|$ and $D = \rho(B) + \sum_{i=n+1}^\infty|\xi_i\rangle\langle\xi_i|$. We have, then,
both operators $C$ and $D$ are invertible on $H$ with $C \leq D$.
Note that 
\begin{align*}
\rho(f(A)^\frac{1}{2})\rho(g(B))\rho(f(A)^\frac{1}{2}) &= \rho(f(A)^\frac{1}{2}g(B)f(A)^\frac{1}{2})\\
&= \sum_{i=1}^nb_{ij}|\xi_i\rangle\langle\xi_j|.
\end{align*}
We have, then, 
\begin{align*}
f(C) &= \rho(f(A)) + f(1)\sum_{i=n+1}^\infty|\xi_i\rangle\langle\xi_i|\\
f(C)^\frac{1}{2}g(D)f(C)^\frac{1}{2} &= (\rho(f(A)) + f(1)\sum_{i=n+1}^\infty|\xi_i\rangle\langle\xi_i|)^\frac{1}{2}\\
&(\rho(g(B)) + \frac{1}{f(1)}\sum_{i=n+1}|\xi_i\rangle\langle\xi_i|)(\rho(f(A)) + f(1)\sum_{i=n+1}^\infty|\xi_i\rangle\langle\xi_i|)^\frac{1}{2}\\
&= \rho(f(A))^\frac{1}{2}\rho(g(B))\rho(f(A))^\frac{1}{2} + \sum_{i=n+1}^\infty|\xi_i\rangle\langle\xi_i|\\
&= \rho(f(A)^\frac{1}{2}g(B)f(A)^\frac{1}{2}) + \sum_{i=n+1}^\infty|\xi_i\rangle\langle\xi_i|\\
&= \sum_{i,j=1}^nb_{ij}|\xi_i\rangle\langle\xi_j| + \sum_{i=n+1}^\infty|\xi_i\rangle\langle\xi_i|
\end{align*}

But 
\begin{align*}
\varphi(C) &= \mathrm{Tr}(S_\varphi (\rho(A) + \sum_{i=n+1}^\infty|\xi_i\rangle\langle\xi_i|))\\
&= \sum_{i=1}^n(S_\varphi \rho(A)\xi_i|\xi_i) + \sum_{i=n+1}^\infty\lambda_i\\
&= \sum_{i=1}^na_{ii}\lambda_i + \sum_{i=n+1}^\infty\lambda_i\\
&> \sum_{i=1}^nb_{ii}\lambda_i + \sum_{i=n+1}^\infty\lambda_i\\
&= \sum_{i=1}^n(S_\varphi \rho(f(A)^\frac{1}{2}g(B)f(A)^\frac{1}{2})\xi_i|\xi_i) + 
\sum_{i=n+1}^\infty\lambda_i\\
&= \mathrm{Tr}(S_\varphi (\rho(f(A)^\frac{1}{2}g(B)f(A)^\frac{1}{2})+ \sum_{i=n+1}^\infty|\xi_i\rangle\langle\xi_i|))\\
&= \varphi(f(C)^\frac{1}{2}g(D)f(C)^\frac{1}{2})
\end{align*}

On the contrary, since $0 < C \leq D$ we have from the assumption
\begin{align*}
\varphi(C + D) - \varphi(|C - D|) &\leq 2\varphi(f(C)^\frac{1}{2}g(D)f(C)^\frac{1}{2})\\
2\varphi(C) &\leq 2\varphi(f(C)^\frac{1}{2}g(D)f(C)^\frac{1}{2})\\
\varphi(C) &\leq \varphi(f(C)^\frac{1}{2}g(D)f(C)^\frac{1}{2}),
\end{align*}
and this is a contradiction. 
Therefore, the function $g$ is operator monotone.
\end{proof}

\vskip 3mm

\begin{cor}
Under the same conditions in Theorem~\ref{thm:monotonicity} $f$ is operator monotone on $(0, \infty)$.
\end{cor}

\vskip 3mm

\begin{proof}
This follows from \cite[Corollary 6]{Hansen-Perd-annal}. 
\end{proof}

\end{document}